\documentclass{article}
\usepackage{amssymb}
\usepackage{amsthm}
\usepackage[english]{babel}

\newtheorem{thm}{Theorem}[section]
\newtheorem{cor}[thm]{Corollary}
\newtheorem{lem}[thm]{Lemma}
\newtheorem{prop}[thm]{Proposition}

\title{A Family of Unitary Operators Satisfying a Poisson-type Summation
Formula}
\author{Dmitry Faifman%
\thanks{School of Mathematical Sciences, Tel-Aviv University, Tel Aviv 69978,
Israel. Partially supported by ISF grant 387/09%
}}
\date{}

\begin{document}
\maketitle
\begin{abstract}
We consider a weighted form of the Poisson summation formula. We prove that under certain decay rate conditions on the weights, there exists a unique unitary Fourier-Poisson operator which satisfies this formula. We next find the diagonal form of this operator, and prove that under weaker conditions on the weights, a unique unitary operator still exists which satisfies a Poisson summation formula in operator form. We also generalize the interplay between the Fourier transform and derivative to those Fourier-Poisson operators.
\end{abstract}

\section{Introduction}
The classical summation formula of Poisson states that, for a
well-behaved function $f:\mathbf{R}\rightarrow\mathbf{C}$ and its
(suitably scaled) Fourier Transform $\hat{f}$ we have the relation
\[\sum_{n=-\infty}^{\infty}f(n)=\sum_{n=-\infty}^{\infty}\hat{f}(n)\]
Fix $x>0$, and replace $f(t)$ with $\frac{1}{x}f(t/x)$.
The Poisson formula is linear and trivial for odd $f$, so we assume
$f$ is even. Also, assume $f(0)=0$. Then
\begin{equation}\sum_{n=1}^{\infty}\hat
f(nx)=\frac{1}{x}\sum_{n=1}^{\infty}f(n/x)\end{equation} We
discussed in \cite{Faifman} the extent to which this summation formula, which
involves sums over lattices in $\mathbf R$, determines the Fourier
transform of a function. Taking a weighted form of the Poisson
summation formula as our starting point, we define a generalized
Fourier-Poisson transform, and show that under certain conditions it
is a unitary operator on $L^2[0,\infty)$. As a sidenote, we show a
peculiar family of unitary operators on $L^2[0, \infty)$ defined by
series of the type $f(x)\mapsto \sum a_n f(nx)$.

\section{Some Notation, and a summary of results}
The Fourier transform maps odd functions to odd functions, rendering The Poisson summation formula trivial. Thus we only consider square-integrable even functions, or equivalently, all functions belong to
$L^2[0,\infty)$. \\Denote by $\delta_n$, $n\geq1$ the sequence given
by $\delta(1)=1$ and $\delta(n)=0$ for $n>1$, and the convolution of
sequences as $(a\ast b)_k=\sum_{mn=k}a_mb_n$.
\\Define the (possibly unbounded) operator
\begin{equation}\label{def1}T(a_n)f(x)=\sum_{n=1}^\infty a_nf(nx)\end{equation} It holds that
$T_{b_n}T_{a_n}f=T_{a_n\ast b_n}f$ whenever the series in both sides
are well defined and absolutely convergent.
\\Let $a_n$, $b_n$, $n\geq1$ be two sequences, which satisfy  $a\ast
b=\delta$. \\This is equivalent to saying that $L(s;a_n)L(s;b_n)=1$
where $L(s;c_n)=\sum_{n=1}^\infty\frac{c_n}{n^s}$. For a given
$a_n$ with $a_1\neq0$, its convolutional inverse is uniquely defined via those
formulas. \\Then, the formal inverse transform to $T_{a_n}$  is
given simply by $T_{b_n}$. \\Note that the convolutional inverse of
the sequence $a_n=1$ is the M\"{o}bius function $\mu(n)$,
defined as
\[\mu(n)=\left\{\begin{array}{ll}(-1)^{\sharp\{p|n \mbox{ prime}\}},& n\mbox{ square-free}\\
                                    0,& d^2|n\end{array}\right.\]
Also, define the operator
\[Sf(x)=\frac{1}{x}f\left(\frac{1}{x}\right)\] - a unitary involution
on $L^2[0,\infty)$, which is straightforward to check.
\\\\In terms of $S$ and $T$, the Poisson summation formula for the Fourier
Transform can be written as following:
\[T(e_n)\hat{f}(x)=ST(e_n)f\]
where $e_n=1$ for all $n$. This suggests a formula for the Fourier
transform:
\begin{equation}\label{Fourier_definition}\hat{f}(x)=T(\mu_n)ST(e_n)f(x)\end{equation}
with $\mu_n=\mu(n)$ the M\"{o}bius function. We would like to mention
that Davenport in \cite{Davenport} established certain identities,
such as
\[\sum_{n=1}^{\infty}\frac{\mu(n)}{n}\{nx\}=-\frac{1}{\pi}\sin(2\pi x)\]
which could be used to show that formula (\ref{Fourier_definition})
actually produces the Fourier transform of a (zero-integral) step
function.
\\\\We define the (possibly unbounded)
Fourier-Poisson Transform associated with $(a_n)$ as
\begin{equation}\mathcal{F}(a_n)f(x)=T(\overline{a_n})^{-1}ST(a_n)f(x)\end{equation}
This is clearly an involution, and produces  the operator
$Sf(x)=\frac{1}{x}f\left(\frac{1}{x}\right)$  for $a_n=\delta_n$,
and (non-formally) the Fourier Transform for $a_n=1$.
Note that both are unitary operators on $L_2[0,\infty)$.
In the following, we will see how this definition can be carried out
rigorously. First we give conditions on $(a_n)$ that produce a unitary operator satisfying a pointwise Poisson summation formula, as was the case with Fourier transform (Theorem \ref{fourier1}). Then we relax the conditions, which produces a unitary operator satisfying a weaker operator-form Poisson summation formula (Theorem \ref{fourier2}).
\\\\\textit{Remark.} A similar approach appears in \cite{Baez},\cite{Burnol} and
\cite{Duffin}, where it is used to study the Fourier Transform and
certain variants of it.

\section{The Fourier-Poisson operator is unitary}
We prove that under certain rate-of-growth assumptions on the
coefficients $a_n$ and its convolution-inverse $b_n$, it holds that $\mathcal F(a_n)=T(\overline{a_n})^{-1}ST(a_n)$ is unitary.\\In the following, $f(x)=O(g(x))$ will be understood to mean at $x\rightarrow\infty$ unless otherwise indicated.
\begin{lem}\label{bounded_lemma} Assume
\begin{equation}\label{init_condition}\sum\frac{|a_n|}{\sqrt n}<\infty\end{equation}
holds, and let $f\in C(0,\infty)$ satisfy $f=O(x^{-1-\epsilon})$ for some $\epsilon>0$. Then $T(a_n)f$ as defined in (\ref{def1}) is a continuous function satisfying $T(a_n)f(x)=O(x^{-1-\epsilon})$. Moreover, $T(a_n)$ extends to a bounded operator on $L^2[0, \infty)$, and
$\|T\|\leq\sum\frac{|a_n|}{\sqrt n}$. \end{lem}
\begin{proof} Consider a
continuous function $f=O(x^{-1-\epsilon})$. It is straightforward to verify that $T(a_n)f$ is well-defined, continuous and $T(a_n)f(x)=O(x^{-1-\epsilon})$. Now apply Cauchy-Schwartz: \[|\langle f(mx), f(nx)\rangle|\leq \frac{1}{\sqrt
 {mn}}\|f\|^2\]
 implying
 \[\|T(a_n)f\|^2\leq\sum_{m,n}|a_m||a_n||\langle f(mx),
 f(nx)\rangle|\leq\left(\sum_n\frac{|a_n|}{\sqrt
 n}\right)^2\|f\|^2\]  So $T(a_n)$ can be extended as a bounded operator to all $L^2$, and $\|T\|\leq\sum\frac{|a_n|}{\sqrt n}$.
 \end{proof}
Now, consider a sequence $a_n$ together with its
convolution-inverse $b_n$. In all the following, we assume that
$a_n$, $b_n$ both satisfty (\ref{init_condition}) (as an example, consider
$a_n=n^{-\lambda}$ and $b_n=\mu(n)n^{-\lambda}$ with $\lambda>0.5$).
\\Then $T(a_n)$,
$T(\overline{b_n})$ are both bounded linear operators, and we define the
Fourier-Poisson operator \[\mathcal F(a_n)=T(\overline{b_n})ST(a_n)\] Note
that $T(a_n)^{-1}=T(b_n)$ (and likewise $T(\overline{a_n})^{-1}=T(\overline{b_n})$), which is easy to verify on the dense subset of continuous functions with compact support.
\begin{cor}\label{pointwise}   Assume that $\sum|a_n|n^\epsilon<\infty$ for some $\epsilon>0$ and $(b_n)$ satisfies (\ref{init_condition}). Take a continuous $f$ satisfying $f(x)=O(x^{-1-\epsilon})$ as $x\rightarrow\infty$ and  $f(x)=O(x^\epsilon)$ as $x\rightarrow0$ for some $\epsilon>0$. Then\\
(a) $\mathcal F(a_n)f$ is continuous and $\mathcal F(a_n)f(x)=O(x^{-1-\epsilon})$.\\
(b) The formula $\sum \overline{a_n}\mathcal F(a_n)f(nx)=(1/x)\sum a_n f(n/x)$ holds pointwise.\end{cor}
\begin{proof} (a) It is easy to see that all the properties of the function are preserved by applying $T(a_n)$ (using $\sum|a_n|n^\epsilon<\infty$) and then by $S$. Then by Lemma \ref{bounded_lemma} application of $T(\overline{b_n})$ to $ST(a_n)$ completes the proof.
\\(b) By Lemma \ref{bounded_lemma}, we get an equality a.e. of two continuous functions: \[T(\overline{a_n})\mathcal F(a_n)f=ST(a_n)f\qedhere\]
\end{proof}
\begin{thm}\label{fourier1} Assume that $\sum|a_n|n^\epsilon<\infty$ for some $\epsilon>0$ and $(b_n)$ satisfies (\ref{init_condition}). Then
$\mathcal F(a_n)$ is a unitary operator.\end{thm}
\begin{proof} Consider
$G=ST(\overline{b_n})ST(a_n)$. Take a continuous function $f$ which is
compactly supported. Define $g(x)=ST(a_n)f(x)=\frac{1}{x}\sum a_n
f(\frac{n}{x})$, and note that $g$ vanishes for small values of $x$, and $|g(x)|=O\left(\sum|a_n|n^\epsilon x^{-1-\epsilon}\right)$.
 Then $T(\overline{b_n})g$ is given by the series (\ref{def1}), and
we obtain the absolutely convergent formula
\[Gf(x)=\sum_{m,n}\frac{a_n\overline{b_m}}{m}f(\frac{n}{m}x)\]
Take two such $f_1, f_2$ and compute
\[\langle Gf_1, Gf_2\rangle=\sum_{k,l,m,n}\frac{a_n\overline{b_m}\overline{a_k}b_l}{ml}\langle f_1(\frac{n}{m}x), f_2(\frac{k}{l}x)\rangle\]
the series are absolutely convergent when both $a_n$ and $b_n$
satisfy (\ref{init_condition}). Now we sum over all co-prime
$(p,q)$, such that $\frac{n}{m}=\frac{p}{q}\frac{k}{l}$. So,
$\frac{nl}{mk}=\frac{p}{q}$, i.e. $nl=up$ and $mk=uq$ for some
integer $u$. Then
\[\langle f_1(\frac{n}{m}x), f_2(\frac{k}{l}x)\rangle=\frac{l}{k}\langle f_1(\frac{p}{q}x), f_2(x)\rangle\]
\[\langle Gf_1,
Gf_2\rangle=\sum_{(p,q)=1}\langle f_1(\frac{p}{q}x),
f_2(x)\rangle\sum_{u}\sum_{mk=uq}\sum_{nl=up}\frac{a_n\overline{b_m}\overline{a_k}b_l}{mk}=\]
\[=\sum_{(p,q)=1}\frac{1}{q}\langle f_1(\frac{p}{q}x),
f_2(x)\rangle\sum_{u}\frac{1}{u}\sum_{mk=uq} \overline{a_kb_m}\sum_{nl=up}
a_nb_l\] and so the only non zero term corresponds to $p=q=1$, $u=1$,
$m=n=k=l=1$, i.e. \[\langle Gf_1,
Gf_2\rangle=\langle f_1, f_2\rangle\] Since $\mathcal F(a_n)$ is
invertible, we conclude that $\mathcal F(a_n)=SG$ is unitary. \end{proof}
\section{An example of a unitary operator defined by series}

Let $a_n\in\mathbb C$ be a sequence satisfying
(\ref{init_condition}). We denote by $C_0(0, \infty)$ the space of
compactly supported continuous functions.
\\Let $T(a_n):C_0(0, \infty)\rightarrow C_0(0, \infty)$ be given by
$(\ref{def1})$.
We will describe conditions on $a_n$ that would imply $\langle
Tf,Tg\rangle_{L^2}=\langle f,g\rangle_{L^2}$ for all $f,g\in C_0(0,
\infty)$. Then we can conclude that $T$ is an isometric operator on
a dense subspace of $L^2[0,\infty)$, and thus can be extended as an
isometry of all $L^2[0,\infty)$.\\A $C$-isometric (correspondingly, unitary) operator will mean an
isometric (unitary) operator, scaled by a constant factor $C$. 
\begin{prop}\label{unitary_sum} $\langle
Tf,Tg\rangle_{L^2}=C^2\langle f,g\rangle_{L^2}$ for all $f,g\in
C_0(0, \infty)$ if and only if for all co-prime pairs $(m_0, n_0)$
\begin{equation}\label{sequence_condition}\sum_{k=1}^{\infty}\frac{a_{m_0
k}\overline{a_{n_0k}}}{k}=\left\{\begin{array}{ll}C^2,& m_0=n_0=1\\
0,& m_0\neq n_0\end{array}\right.\end{equation}\end{prop}
\begin{proof} Take $\epsilon>0$.  Denote $M=\sup\{x|f(x)\neq 0
\vee g(x)\neq0\}$. Write
\[\int_\epsilon^\infty Tf(x)\overline{Tg(x)}dx=\int_\epsilon^\infty\sum_{m,n=1}^\infty a_m\overline{a_n}f(mx)\overline{g(nx)}dx\]
It is only necessary to consider $m, n<M/\epsilon$. Thus the sum is
finite, and we may write
\[\int_\epsilon^\infty Tf(x)\overline{Tg(x)}dx=\sum_{m,n=1}^\infty
a_m\overline{a_n}\int_\epsilon^\infty f(mx)\overline{g(nx)}dx\] Note
that
\[\int_0^\infty f(mx)\overline{g(nx)}dx\leq\|f(mx)\|\|g(nx)\|=\frac{1}{\sqrt{mn}}\|f\|\|g\|\]
and therefore \[\sum_{m,n=1}^\infty a_m\overline{a_n}\int_0^\infty
f(mx)\overline{g(nx)}dx\] is absolutely convergent:
\[\sum_{m,n=1}^\infty\left|a_m\overline{a_n}\int_{0}^\infty
f(mx)\overline{g(nx)}dx\right|\leq\sum_{m,n=1}^\infty\frac{|a_m||a_n|}{\sqrt{mn}}\|f\|\|g\|\]
Therefore, the sum
\[S(\epsilon)=\sum_{m,n=1}^\infty a_m\overline{a_n}\int_{0}^{\epsilon}
f(mx)\overline{g(nx)}dx\] is absolutely convergent. We will show
that $S(\epsilon)\rightarrow0$ as $\epsilon\rightarrow\infty$.
Assume $|f|\leq A_f$, $|g|\leq A_g$. Then
\[|S(\epsilon)|\leq \sum_{m,n}a_m\overline{a_n}\int_0^\epsilon
f(mx)\overline{g(nx)}dx\leq\sum_{m,n}\frac{|a_ma_n|}{\sqrt
{mn}}\sqrt{\int_0^{m\epsilon} |f|^2}\sqrt{\int_0^{n\epsilon}
|g|^2}\] And
\[\sum_{m=1}^\infty\frac{|a_m|}{\sqrt m}\sqrt{\int_0^{m\epsilon}
|f|^2}=\sum_{m=1}^{\sqrt{1/\epsilon}}+\sum_{m=\sqrt{1/\epsilon}}^\infty\leq
\sum_{m=1}^\infty \frac{|a_m|}{\sqrt m}\sqrt{\int_0^{\sqrt\epsilon}
|f|^2} + \|f\|\sum_{m=\sqrt{1/\epsilon}}^\infty \frac{|a_m|}{\sqrt
m}\longrightarrow 0\]
 We conclude that
\[\langle Tf, Tg\rangle=\sum_{m,n=1}^\infty
\frac{a_m\overline{a_n}}{m}\int_{0}^\infty
f(x)\overline{g(\frac{n}{m}x)}dx=\left\langle f(x),
\sum_{(m_0,n_0)=1}\frac{1}{m_0}\sum_{k=1}^{\infty}\frac{a_{m_0
k}\overline{a_{n_0k}}}{k} g(\frac{n_0}{m_0}x)\right\rangle\]
Therefore, $T(a_n)$ is a $C$-isometry on $C_0(0,\infty)$ if and only
if $(a_n)$ satisfy (\ref{sequence_condition}). \end{proof}
\noindent\textbf{Example 1.} Take $a_n^{(2)}=0$ for $n\neq 2^k$ and
\[a_{2^k}^{(2)}=\left\{\begin{array}{ll}1,& k=0\\
(-1)^{k+1},& k\geq1\end{array}\right.\] Then
\[T_2f(x)=\sum_n a_n^{(2)}f(nx)=f(x)+f(2x)-f(4x)+f(8x)-f(16x)+...\]
is a $\sqrt 2$-isometry.
\\\\\textbf{Example 2.} Generelazing example 1 (and using the already defined $a_n^{(2)}$), we fix a natural number $m$, and take
$a_{m^k}^{(m)}=\left(\frac{m}{2}\right)^{k/2}a_{2^k}^{(2)}$ and
$a_n^{(m)}=0$ for $n\neq m^k$. Then \[T_m f(x)=\sum
a_n^{(m)}f(nx)\] is again a $\sqrt 2$-isometry.
\\\\\textbf{Example 3.} Similiarly, we could take $a_n=0$ for $n\neq 2^k$ and
\[a_{2^k}=\left\{\begin{array}{ll}1,& k=0\\
-1,& k\geq1\end{array}\right.\] Then
\[Tf(x)=f(x)-f(2x)-f(4x)-f(8x)-f(16x)-...\]
is a $\sqrt 2$-isometry.
\\\\\textit{Remarks}
\\$\bullet$ If $a_n$ and $b_n$ satisfy
(\ref{init_condition}),
then so does their convolution $c_n=(a\ast b)_n=\sum_{kl=n}a_k b_l$:
\[\sum_n\frac{|c_n|}{\sqrt n}\leq \sum_n\sum_{kl=n}\frac{|a_k||b_l|}{\sqrt{kl}}
=\sum_k\frac{|a_k|}{\sqrt k}\sum_l\frac{|b_l|}{\sqrt l}<\infty\]
\\$\bullet$ Also, any two scaled isometries of the form $T(a_n)$
commute: If $a_n$ and $b_n$ satisfy (\ref{sequence_condition}),
$T(a_n)$ and $T(b_n)$ are isometries from $C_0(0,\infty)$ to itself,
and thus so is their composition which is easily computed to be
$T(a_n\ast b_n)$.
\begin{prop}When $(a_n)$ satisfies
(\ref{sequence_condition}), $T(a_n)$ is $C$-unitary.\end{prop}
\begin{proof} It is easy to verify that for any $g\in C[0,\infty)$ and $a_n$
satisfying (\ref{sequence_condition}),
\[T(a_n)^*g=\sum\frac{\overline{a_n}}{n}g(\frac{x}{n})\]
Moreover, $T(a_n)^*$ is a scaled isometry on $\{f\in C_0[0,\infty):
supp(f)\subset[a,b], a>0\}$ (proof identical to that of $T(a_n)$),
and so a scaled isometry on $L^2$. Thus $T^*T=TT^*=\|T\|^2I$, and so
$T(a_n)$ is $C$-unitary.\end{proof}
\noindent\textit{Remark.} We recall the operator $Sf(x)=\frac{1}{x}f\left(\frac{1}{x}\right)$ - a unitary operartor of
$L^2(0,\infty)$. Then for a continuous function $f$ with compact
support which is bounded away from 0, we have $Sf\in C_0(0,\infty)$
and so we can use (\ref{def1}) to obtain $ST(\overline{a_n})Sf=T(a_n)^*f$ and
therefore $ST(\overline{a_n})S=T(a_n)^*$ on all $L^2$. In particular, for real sequences $(a_n)$, $ST^m$
and $T^mS$ are unitary involutions (up to scaling) for any integer
$m$.

\section{Diagonalizing  the Fourier-Poisson operator}
We further generalize the Poisson summation formula: by removing some of the conditions on the sequence $(a_n)$, we are still able to construct a unitary operator satisfying the summation formula, but only in the weaker operator sense.  This is done through a natural isometry between $L^2[0,\infty)$ and $L^2(-\infty, \infty)$ which was suggested to us by Bo'az Klartag (see also \cite{Korani}).
\\\\We will denote by $dm$ the Lebesgue measure on $\mathbb R$, and $\hat g$ will stand for the
Fourier transform defined as $\hat{g}(\omega)=\int_{-\infty}^\infty
g(y)e^{-iy\omega}dy$. \\First, define two isometries of spaces:
\\(1) $u:L^2\left([0,\infty), dm(x)\right)\rightarrow
L^2\left(\mathbb R, e^ydm(y)\right)$ given by $f(x)\mapsto
g(y)=f(e^y)$.
\\(2) $v:L^2\left(\mathbb R,
e^ydm(y)\right)\rightarrow L^2\left(\mathbb R, dm\right)$ given by
$g(y)\mapsto h(x)=(2\pi)^{-\frac{1}{2}}\hat{g}(x+i/2)$.
\\$u$ is isometric by a simple change of variables. \\To see that $v$
is isometric, note that $\widehat{f}(x+i/2)=\widehat{e^{t/2}f(t)}(x)$,
and so by Plancherel's formula
\[\int|\hat{f}(x+i/2)|^2dx=2\pi\int|f(t)|^2e^t
dt\] (alternatively, one could decompose $v$ into the composition of
two isometries: $f(y)\mapsto e^{y/2}f(y)$, identifying
$L^2\left(\mathbb R, e^ydm(y)\right)$ with $L^2(\mathbb R, dm)$, and
then Fourier transform).
\\ We will denote the composition $v\circ
u=w$.
\\For $A:L^2[0,\infty)\rightarrow L^2[0,\infty)$, we write $\widetilde{A}=wAw^{-1}:L^2(\mathbb R)\rightarrow L^2(\mathbb R)$ -
the conjugate operator to $A$. The conjugate to $S$ is $\tilde
S(h)(x)=h(-x)$.
\\\\Let $a_n$ satisfy (\ref{init_condition}), implying $|L(1/2+ix; a_n)|$ is bounded and continuous. Then for $g=u(f)$,
\[(uT(a_n)u^{-1}g)(y)=\sum a_n g(y+\log n)=g\ast\nu(y)\]
where $\nu(y)=\sum a_n\delta_{-\log n}(y)$ and $\hat{\nu}(z)=\sum
a_n e^{i z\log n}=\sum a_n n^{iz}=L(-iz; a_n)$, which converges for
$Im z\geq1/2$ by (\ref{init_condition}). And so letting $h=vg$,
\[\widetilde{T(a_n)}h(x)=(2\pi)^{-\frac{1}{2}}\widehat{g\ast \nu}(x+i/2)=
L(1/2-ix; a_n)h(x)\]
thus we proved
\begin{cor}\label{unitary_l} Assume $a_n$ satisfies
(\ref{init_condition}). Then the following are equivalent:
\\(a)
$|L(1/2+ix; a_n)|=C$ \\(b) $T(a_n)$ is $C$-unitary on
$L^2[0,\infty)$\\(c) $(a_n)$ satisfies
(\ref{sequence_condition}).\end{cor} \noindent The equivalence of
(a) and (c) can easily be established directly.
\\\\For example, the $\sqrt 2$-unitary $Tf(x)=f(x)+f(2x)-f(4x)+f(8x)-...$ discussed previously is associated with $L(s; a_n)=\frac{2+2^s}{1+2^s}$ which has absolute value of $\sqrt 2$ on $Re (s)=1/2$.
\\\\This suggests that the Fourier-Poisson transform associated with $a_n$,
which was defined in section 3 for some special sequences $(a_n)$, could be generalized as follows:
$\mathcal F(a_n)f=T(\overline{a_n})^{-1}ST(a_n)f$ should be defined through
\begin{equation}\label{Fourier_Poisson}\widetilde{\mathcal F(a_n)}h(x)=h(-x)\frac{L(1/2+ix; a_n)}{L(1/2-ix; \overline{a_n})}=h(-x)\frac{L(1/2+ix; a_n)}{\overline{L(1/2+ix; a_n)}}\end{equation}
we arrive at the following
\begin{thm}\label{fourier2} Assume $\sum |a_n| n^{-1/2}<\infty$. Then
\\(a) There exists a bounded operator $\mathcal F(a_n):L^2[0,\infty)\rightarrow L^2[0,\infty)$
satisfying the Poisson summation formula (in its operator form) $T(\overline{a_n})\mathcal F(a_n)=ST(a_n)$. Moreover, $\mathcal F(a_n)$ is unitary.
\\(b) If for some $\epsilon>0$, $\sum |a_n|n^{-1/2+\epsilon}<\infty$, then a bounded $\mathcal F(a_n)$ satisfying $T(\overline{a_n})\mathcal F(a_n)=ST(a_n)$ is unique.\end{thm}
\begin{proof} (a) We have $L(1/2+ix; a_n)/\overline{L(1/2+ix; a_n)}=e^{2i(\arg L(1/2+ix; a_n))}$ whenever $L(1/2+ix;a_n)\neq 0$.
In accordance with (\ref{Fourier_Poisson}), define \[\widetilde{\mathcal F(a_n)}h(x)=e^{2i(\arg L(1/2+ix; a_n))}h(-x)\]
taking $\arg L(1/2+ix; a_n)=0$  whenever $L(1/2+ix; a_n)=0$.
We then have
\[L(1/2-ix; \overline{a_n})\widetilde{F(a_n)}h(x)=L(1/2+ix; a_n)h(-x)\]
 for all
$h\in L^2(\mathbb R)$, implying
$T(\overline{a_n})\mathcal F(a_n)=ST(a_n)$ in $L^2[0, \infty)$. Also, $\mathcal F(a_n)$ is isometric and invertible, thus unitary.
\\(b) For uniqueness, observe that $L(s;a_n)$ is analytic in a neighborhood of $Re(s)=1/2$, and so its set of zeros $Z$ is discrete,
and the ratio $L(1/2+ix; a_n)/\overline{L(1/2+ix; a_n)}$ is continuous and of absolute value 1 outside of $Z$. Thus for continuous $h$ with $supp(h)\cap Z=\emptyset$, the equation
\[L(1/2-ix; \overline{a_n})\widetilde{F(a_n)}h(x)=L(1/2+ix)h(-x)\] determines $\widetilde{F(a_n)h}$ uniquely, and all such $h$ are dense in $L^2(\mathbb R)$.
\end{proof}
\noindent By part (b) we conclude that under the conditions of Theorem \ref{fourier1}, the operator $\mathcal F(a_n)$ defined in section 3 coincides with the operator defined here.
\noindent
\\\\\textit{Remark.} It was pointed out to us by Fedor Nazarov that under the conditions of Theorem \ref{fourier2} the Poisson summation formula cannot hold pointwise for all sequences $(a_n)$.

\section{A formula involving differentiation}
\noindent Denote by $B:L^2[0,\infty)\rightarrow L^2[0,\infty)$ the unbounded operator
\[Bf(x)=i(xf'+f/2)\] with $Dom(B)=\{f\in C^\infty: xf'+f/2\in L^2\}$. It is straightforward to check that $B$ is a symmetric operator.
\\\\
It is easy to verify that the ordinary Fourier transform $\mathcal F$ satisfies, for a well behaved (i.e. Schwartz) function $f$, the identity $B\mathcal F f+\mathcal FB f=0$. It turns out to be also a consequence of Poisson's formula, and so holds for a large family of operators. We will need the following standard lemma (see \cite{Simon})
\begin{lem}\label{schwartz} Take a function $g\in L^2(\mathbb R)$. The following are equivalent:
\\(a) $g\in C^\infty(\mathbb R)$ and \[\sup _{|y|<b}\sup_t e^{yt}|g^{(k)}(t)|<\infty\]
for all $b<B$ and $k\geq 0$.
\\ (b) $h=\hat g$ is a Schwartz function, which has an analytic extension to the strip $|y|<B$ such that \[\sup_{|y|<b}\sup_x |x|^k |h(x+iy)|<\infty\] for all $b<B$ and $k\geq 0$.
\end{lem}
\begin{proof}
\textit{(a)$\Rightarrow$(b)}. Observe that $\hat g(x+iy)=\widehat{e^{yt}g(t)}(x)$. Thus the existence of analytic extension is clear, and we can write
\[ |x|^k |h(x+iy)|=|x|^k|\widehat{e^{yt}g(t)}(x)|=|\widehat{\left(e^{yt}g(t)\right)^{(k)}}(x)|\]
Note that \[\left(e^{yt}g(t)\right)^{(k)}=e^{yt}\sum_{j=0}^k P_{j,k}(y) g^{(j)}(t)\]
where $P_{j,k}$ denotes some universal polynomial of degree $\leq k$.
Therefore
\[ \sup_x |x|^k |h(x+iy)|\leq \int_{-\infty}^\infty\left|e^{yt}\sum_{j=0}^k  P_{j,k}(y) g^{(j)}(t)\right|dy\]
The sum is finite, so we can bound every term separately. Choose $b<Y<B$, $\epsilon=Y-b$. Then
\[\sup_{|y|<b}\int_{-\infty}^\infty|e^{yt}g^{(j)}(t)|\leq C(j, Y)\int_{-\infty}^\infty e^{-\epsilon |t|}dt<\infty\]
\textit{(b)$\Rightarrow$(a)}. Note that $g$ is a Schwartz function since $h$ is. It sufficies to show (by induction) that supremums of $|(e^{yt}g)^{(k)}|$ are finite for every $k$ and $b<B$. Notice that $\widehat{g^{(k)}}(x)=(ix)^k\hat g(x)$ has an analytic extension to the strip $|y|<B$ (namely: $(iz)^kh(z)$), satisfying the same conditions as $h$ itself.
Now take a $C^\infty$ compactly supported function $\phi$ on $\mathbb R$.
We will show that
\begin{equation}\label{paley}\int_{-\infty}^\infty e^{yt}g(t)\overline{\phi(t)}dt=\int_{-\infty}^\infty h(x+iy)\overline{\hat \phi(x)}dx\end{equation}
implying \[\widehat{e^{yt}g(t)}(x)=h(x+iy)\] and therefore for any $k$
\[\widehat{e^{yt}g^{(k)}(t)}(x)=i^k(x+iy)^kh(x+iy)\] which is equivalent to having
\[\widehat{\left(e^{yt}g(t)\right)^{(k)}}(x)=i^kx^kh(x+iy)\]
Indeed, $\psi=\hat\phi$ is an analytic function satisfying the supremum condition by the "$(a)\Rightarrow(b)$" implication. Then
\[\int_{-\infty}^\infty e^{yt}g(t)\overline{\phi(t)}=\int_{-\infty}^\infty \hat g\overline{\widehat{e^{yt}\phi(t)}}dt=
\int_{-\infty}^\infty h(x)\overline{\psi(x+iy)}dt\]
Observe that $\lambda(z)=h(z)\overline{\psi(iy+\overline z)}$ is an analytic function, and the integrals over the intervals $Re (z)=\pm R$, $-b<Im(z)<b$ of $\lambda(z)$ converge to 0 as $R\rightarrow\infty$ by the uniform bounds on $h$ and $\psi$. Considering the line integral of $\lambda$ over a rectangle with these vertical sides and horizontal lines at $Im(z)=0$ and $Im(z)=y$, we get
\[\int_{-\infty}^\infty h(x)\overline{\psi(iy+x)}=\int_{-\infty}^\infty h(x+iy)\overline{\psi(x)}\]
which proves (\ref{paley}). Finally,
\[\sup_{|y|<b}\sup_t |(e^{yt}g(t))^{(k)}|\leq\sup_{|y|<b}\int_{-\infty}^\infty |x|^kh(x+iy)dx\]
which is finite by the assumptions.\end{proof}
\noindent
Let $\mathcal S_0$ be the following class of "Schwartz" functions in  $L^2[0,\infty)$
\[\mathcal S_0=\{f\in C^\infty: \sup |x|^n|f^{(k)}(x)|<\infty \mbox{ }\forall k\geq0\mbox{ ,}\forall n\in\mathbb Z\}\]
Note that $n\in\mathbb Z$ can be negative. Observe that $\mathcal
S_0\subset Dom(B)$.
\begin{prop} Assume $(a_n)$ satisfies $\sum |a_n|n^\epsilon<\infty$ for some $\epsilon>0$, and the convolution inverse $(b_n)$ satisfies $\sum |b_n|/\sqrt n<\infty$. Next, assume that \[L(1/2+iz; a_n)/\overline{L(1/2+i\overline {z};a_n)}\] (which is meromorphic by assumption in the strip $|y|<1/2+\epsilon$) satisfies the following polynomial growth condition:  there exist constants $N$ and $C$ such that
\[\left|\frac{L(1/2-y+ix; a_n)}{L(1/2+y+ix;a_n)}\right|\leq C_0+C_1 |x|^{N}\] for all $x,y\in\mathbb R$, $|y|\leq 1/2+\epsilon/2$.
Let $f\in \mathcal S_0$. Then $\mathcal F(a_n) B f + B\mathcal F(a_n) f=0$.\end{prop}
\begin{proof}  Denote $g=\mathcal F(a_n)f$. Denote $F(t)=e^{t/2}f(e^t)$ and $G(t)=e^{t/2}g(e^t)$, $h_f=\hat F$ and $h_g=\hat G$. The condition $f\in \mathcal S_0$ implies immediately that $F\in C^\infty$ and $\sup_{t\in\mathbb R} e^{yt}|F^{(k)}(t)|<\infty$ for all $y\in\mathbb R$, since $F^{(k)}(t)=P_k(e^{t/2}, f(e^t),...,f^{(k)}(e^t))$ for some fixed polynomial $P_k$. By Lemma (\ref{schwartz}), $h_f$ is a Schwartz function (on the real line), with an analytic extension to the strip $|y|<1$ such that \[\sup_{|y|\leq1}\sup_x |x|^k |h_f(x+iy)|<\infty\] for all $k\geq 0$. Next,
\[h_g(x+iy)=\frac{L(1/2-y+ix; a_n)}{\overline{L(1/2+y+ix;a_n)}}h_f(x+iy)\]
is an analytic function in the strip $|y|<1/2+\epsilon$. By the assumed bound on the L-function ratio, it is again a Schwartz function when restricted to the real line; and  \[\sup_{|y|<b}\sup_x |x|^k |h_g(x+iy)|<\infty\] for all $b<1/2+\epsilon/2$. Denote $\delta=\epsilon/4$. Again by Lemma (\ref{schwartz}), $G\in C^\infty$  and
satisfies $e^{(1/2+\delta )|t|}|G(t)|\leq C \iff |G(t)|\leq C e^{-(1/2+\delta)|t|} $
and likewise $ |G'(t)|\leq C e^{-(1/2+\delta)t}$ for some constant
$C$.  Then, as $t\to-\infty$, \[|g(e^t)|=e^{-t/2}|G(t)|\leq C e^{-t/2}  e^{(1/2+\delta)t}=O(e^{-\delta t})\]  and as $t\to+\infty$,
\[|g(e^t)|=e^{-t/2}|G(t)|\leq Ce^{-t/2}e^{-(1/2+\delta)t}=O(e^{-(1+\delta)t})\]Also, as $t\to+\infty$, \[|g'(e^t)|=\left|G'(t)-\frac{1}{2}G(t)\right|e^{-3t/2}=
O(e^{-(2+\delta)t})\]Thus $g\in C^\infty(0,\infty)$ and $g=O(x^\delta)$ as $x\rightarrow 0$,  $g=O(x^{-1-\delta})$ as $x\rightarrow \infty$ while
$g'=O(x^{-2-\delta})$ as $x\rightarrow\infty$. By Corollary \ref{pointwise} (b) we can write $\sum \overline{a_n}g(nx)=(1/x) \sum a_n f(n/x)$, and then the functions on both sides are $C^1$, and can be differentiated term-by-term. Carrying the differentiation out, we get
\[\sum \overline{a_n} (nx)g'(nx)=-(1/x)\sum a_n f(n/x)-(1/x^2)\sum a_n(n/x)f'(n/x)\]
Invoke Lemma \ref{bounded_lemma} to write
\[T(\overline{a_n})(xg')= -T(\overline{a_n})g-ST(a_n) (xf')\]
and then use Corollary \ref{pointwise} applied to $xf'$ to conclude
\[T(\overline{a_n})(xg')= -T(\overline{a_n})g-T(\overline{a_n})\mathcal F(a_n) (xf')\]
Finally, apply $T(\overline b_n)$ to obtain the announced result.
\end{proof}
\noindent\textit{Remark.} As an example of such a sequence, take $a_n=n^\lambda$, $\lambda<-1$.

\section{Acknowledgements}
I am indebted to Bo'az Klartag for the idea behind section 5, and
also for the motivating conversations and reading the drafts. I am
grateful to Nir Lev, Fedor Nazarov, Mikhail Sodin and Sasha Sodin
for the illuminating conversations and numerous suggestions. Also,
I'd like to thank my advisor, Vitali Milman, for the constant
encouragement and stimulating talks. Finally, I would like to thank
the Fields Institute for the hospitality during the final stages of
this work.

\end{document}